\documentclass[10pt]{amsart}
\usepackage{array,youngtab}
\usepackage{color, graphicx, comment}
\usepackage{tikz}
\usetikzlibrary{automata,positioning}

\usepackage{mathtools}

\definecolor{darkgreen}{rgb}{0.,0.5,0.}

\allowdisplaybreaks

\numberwithin{equation}{section} \overfullrule 5pt
\newtheorem{thm}{Theorem}[section]

\newtheorem{lem}[thm]{Lemma}

\theoremstyle{definition}
\newtheorem{defi}{Definition}[section]

\newcommand{\inv}{\text{inv}}
\newcommand{\NN}{\mathbb{N}}
\newcommand{\Sym}{\mathfrak{S}}

\title[Automaticity of the Hankel determinants]{%
On the automaticity of the Hankel determinants of a family of automatic sequences} 
\date{August 12, 2018}

\author{Yining Hu}
\address{School of Mathematics and Statistics, 
Huazhong University of Science and Technology, Wuhan, PR China}
\email{huyining@protonmail.com}

\author{Guoniu Wei-Han}
\address{I.R.M.A., UMR 7501, Universit\'e de Strasbourg
et CNRS, 7 rue Ren\'e Descartes, F-67084 Strasbourg, France}
\email{guoniu.han@unistra.fr}

\subjclass[2010]{05A05, 11B50, 11B85, 11C20 11J72, 11J82}

\keywords{automatic sequence, Hankel determinant, irrationality exponent}

\begin{document}
\begin{abstract} 
Hankel determinants and automatic sequences are two classical subjects
widely studied in Mathematics and Theoretical Computer Science.
However, these two topics were considered totally independently, until in 1998,
when Allouche, Peyri\`ere, Wen and Wen proved that all the Hankel determinants 
of the Thue-Morse sequence are nonzero. 
This property allowed Bugeaud to prove that the irrationality exponents of the Thue-Morse-Mahler numbers are exactly 2.
Since then, the Hankel determinants
of several other automatic sequences, in particular, the paperfolding sequence,
the Stern sequence, the period-doubling sequence, are studied by 
Coons, Vrbik, Guo, Wu, Wen, Bugeaud, Fu, Han,  
Fokkink, Kraaikamp, and Shallit.
On the other hand, it is known that the Hankel determinants of a rational power series
are ultimately zero, and
the Hankel determinants of a quadratic power series over finite fields are ultimately periodic.
It is therefore natural to ask if we can obtain similar results about the Hankel determinants of algebraic series. 
In the present paper, we provide a partial answer to this question by establishing the automaticity of the reduced Hankel determinants modulo $2$
of a family of automatic sequences.
As an application of our result, we give upper bounds for the irratoinality
exponent of a family of automatic numbers.
\end{abstract}

\maketitle

\section{Introduction}\label{sec:Intro} 

The present paper deals with two classical objects in Mathematics
and Theoretical Computer Science, namely, Hankel determinants
and automatic sequences.
For each infinite sequence
${\bf c}=(c_j)_{j\geq 0}$
and each nonnegative integer~$n$ 
the {\it Hankel determinant} of order $n$ of 
the sequence ${\bf c}$ is defined by
\begin{equation}\label{def:Hankel}
H_{n}({\bf c}):=
\begin{vmatrix}
  c_0 & c_{1}&\cdots & c_{n-1} \\
  c_{1} & c_{2}& \cdots & c_{n}\\
   \vdots & \vdots &\ddots & \vdots\\
   c_{n-1} & c_{n} & \cdots & c_{2n-2}
\end{vmatrix}.
\end{equation}
We also speak of the Hankel determinants of the power series 
$\tilde{\bf c}(x)=\sum_{k\geq 0} c_j x^j$ and write
$H_n(\tilde{\bf c}(x)) = H_n({\bf c})$.
The Hankel determinants are widely studied in Mathematics and,
in several cases, can be evaluated by  
basic determinant manipulation, $LU$-decomposition, or  Jacobi 
continued fraction (see, e.g., \cite{Kr98, Kr05, Fl80}).
In this article, we consider the {\it reduced Hankel determinants} of $\pm1$-sequences, defined as $H_n({\bf c})/2^{n-1}$, since the Hankel determinant of order $n$ of a $\pm 1$-sequence is always divisible by $2^{n-1}$.

On the other hand,  a sequence is said to be {\it  $d$-automatic} if it can be generated by a 
$d$-DFAO ({\it deterministic finite automaton with output})\cite{AS2003}. 
For an integer $d\geq 2$,  a $d$-DFAO is defined to be a $6$-tuple
$$M=(Q,\Sigma, \delta, q_0, \Delta, \tau)$$
where $Q$ is the set of states with $q_0\in Q$ being the initial state, $\Sigma=\{0,1,\ldots,d-1\}$ the input alphabet, $\delta:Q\times \Sigma\rightarrow Q$ the transition function, $\Delta$ the output alphabet, and $\tau:Q\rightarrow \Delta$ the output function. 
The $d$-DFAO $M$ generates a sequence $(c_n)_{n\geq 0}$ in the following way: for each non-negative integer $n$, the base-$d$ expansion of $n$ is read by $M$ from right to left starting from the initial state $q_0$, and the automaton moves from state to state according to its transition funciton $\delta$. When the end of the string is reached, the automaton halts in a state $q$, and the automaton outputs the symbol $c_n=\tau (q)$. 
In practice, an automaton is often presented as a multigraph with states as vertices and the transition function as directed edges labelled by the input alphabet. 

Here is an example of a $2$-automaton:

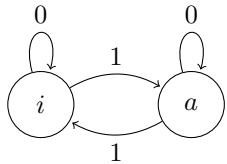
\begin{center}\begin{tikzpicture}[shorten >=1pt,node distance=2cm,on grid,auto] 
   \node[state] (i)   {$i$}; 
   \node[state] (a) [right=of i]{$a$}; 
    \path[->] 
    (i) edge [loop above] node  {0} (i)
          
    (i) edge [bend left] node {1} (a)
         
    (a) edge [loop above] node  {0} (a)

    (a) edge [bend left] node  {1} (i);
\end{tikzpicture}

$\tau(i)=0,\; \tau(a)=1.$
\end{center}
The sequence $(t(n))_n$ generated by this automaton is the {\it Thue-Morse} sequence 
whose first terms are $0110100110010110\cdots$. 

An equivalent definition for a sequence $\bf u$ to be $d$-automatic is that the {\it $d$-kernel} of $\bf{u}$, defined as
 $$\{ (u(d^k n+j))_{n\geq 0} \mid k\in \mathbb{N},\, 0\leq j\leq d^k-1 \},$$
 is finite \cite[Prop. V.3.3]{Eilenberg1974A}, \cite{Christol1979}.
 If we let $\Lambda_{i}^{(d)}$ denote the operator that sends a 
 sequence $(u(n))_{n\geq 0}$ to its subsequence $(u(d n+i))_{n\geq 0}$,
 then the $d$-kernel can be defined alternatively 
as the smallest set containing $\bf{u}$ that is stable under $\Lambda_i^{(d)}$ for $0\leq i<d-1$. We write $\Lambda_i$ instead of $\Lambda_i^{(d)}$ when the value of $d$ is clear from the context.

Automatic sequences arise naturally in various contexts. The study of automatic sequences lies at the interface of number theory, combinatorics on words, dynamic system, logic and theoretical computer science. 
We refer the readers to  \cite{AS2003} for a comprehensive exposition of the subejct.
One fundamental result is
the link between automatic sequences and algebraicity in positive characteristic, 
discovered by Christol, Kamae, Mend\`es France and Rauzy 1980 \cite{CKMFR1980}. They proved that
a series in $\mathbb{F}_q[[x]]$
is algebraic over $\mathbb{F}_q(x)$ if and  only if the sequence of its coefficients is $q$-automatic.

Hankel determinants and automatic sequences were considered totally independently, until in 1998,
when Allouche, Peyri\`ere, Wen and Wen proved that all the Hankel determinants 
of the Thue-Morse sequence are nonzero \cite{APWW1998}. 
This property allowed Bugeaud to prove that the irrationality exponents of the Thue-Morse-Mahler numbers are exactly 2 \cite{Bu2011}. 
Since then, the Hankel determinants
for several other automatic sequences, in particular, the paperfolding sequence,
the Stern sequence, the period-doubling sequence, are studied by 
Coons, Vrbik, Guo, Wu, Wen, Bugeaud, Fu, Han,  
Fokkink, Kraaikamp, and Shallit \cite{Coons2013, GWW2013, HanWu2015, FKS2017,
Han2015hankel, Han:Hankel, FuHan2015:ISSAC, BH2014}. 

Apart from their link with irrationality exponent, one reason for considering
the Hankel determinants of automatic sequences is the following:
It is known that
a formal power series is rational if and only if its Hankel determinants are
ultimately zero \cite[p.~5, Kronecker Lemma]{Salem1983}, and
the Hankel determinants of a quadratic power series over finite fields are ultimately periodic \cite{Han:Hankel}.
It is therefore reasonable to expect analogous properties from the Hankel determinants of
algebraic series. 

We now introduce the class of sequences that we will study in this article. 
Let $d\geq 2$ be a positive integer and ${\bf v}=(v_0, v_1,v_2,\ldots,v_{d-1})$ a finite $\pm 1$-sequence of length~$d$ such that $v_0=1$.
The generating polynomial of $\bf v$ is denoted by 
$\tilde{\bf v}(x)= \sum_{i=0}^{d-1} v_i x^i $.
We consider the $\pm 1$-sequence $\bf f$ defined by the following power series
\begin{equation}\label{def:Phi}
	{\tilde{\bf f}}(x) :=  \Phi(\tilde{\bf v}(x)) := \prod_{k=0}^\infty \tilde{\bf v} (x^{d^k})
\end{equation}
We see that the $n$-th term of $\bf f$ is equal to
$$
f_n=\prod_{i=1}^{d-1} v_i ^{s_{d,i}(n)},
$$
where $s_{d,i}(n)$ denotes the number of occurrences of 
the digit~$i$ in the base-$d$ representation of $n$.
Thus, the Thue--Morse power series 
is equal to $\Phi(1-x) = \prod_{k=0}^\infty(1-x^{2^k})$. 
Since the power series ${\tilde {\bf f}(x)}$ satisfies the following 
functional equation 
\begin{equation}\label{equ:f:func}
\tilde{\bf f}(x)= 
  \tilde{\bf v}(x)\prod_{k=1}^{\infty}\tilde{\bf v}\left(x^{d^{k}}\right)
  =\tilde{\bf v}(x)\tilde{\bf f}(x^{d}),
\end{equation}
the sequence ${\bf f}$  can also be defined by the 
recurrence relations
\begin{equation}\label{equ:f:recurrence}
  f_0=1,\quad f_{dn+i}=v_if_{n} \quad\text{\ for  $n\geq 0$ and $0\leq i \leq d-1$.}
\end{equation}

\medskip
We prove the following Theorem, which we will use in Section \ref{section:IE}
to obtain upper bounds for the irrationality exponent of a family of automatic
numbers.
\begin{thm}\label{thm:main}
	For each positive integer $d\geq 2$ and each $\pm 1$-vector $\bf v$ of length
	$d$, the sequence of the reduced Hankel determinants modulo 2 of the sequence
	$\bf f$ defined by $\tilde{\bf{f}}=\Phi(\tilde{\bf{v}}(x))$ is $d$-automatic.
\end{thm}

The proof of Theorem \ref{thm:main} is based on some previous works 
about the Hankel determinants developed by Fu-Han \cite{FuHan2015:ISSAC}, see
Section \ref{sec:proof}.
We give two examples 
of Theorem \ref{thm:main} in Section \ref{sec:examples}. The generating series of the two sequences are respectively
$$\prod_{k\geq 0}(1+x^{5^k} - x^{2\cdot5^k}-x^{3\cdot 5^k} + x^{4\cdot5^k})$$
and
$$\prod_{k\geq 0}(1+x^{4^k} - x^{2\cdot4^k}-x^{3\cdot 4^k} ).$$
By Theorem \ref{thm:main},
the sequence of the reduced Hankel determinants modulo~2 is $5$-automatic (resp. $4$-automatic).
The minimal automata are illustrated in Section~\ref{sec:examples}.


\section{Proof of the Main Theorem}\label{sec:proof} 
In this section we first recall some previous works developed in \cite{FuHan2015:ISSAC}
in order to explain the meaning of the notions used in the proof of the Main Theorem.
Then we prove two general Lemmas. Finally we apply these two Lemmas to the results from \cite{FuHan2015:ISSAC} to prove the Main Theorem. 

\subsection{Key notions from previous works}
In the proof of the Main Theorem, we will make use of recurrence relations between
six sequences $\bf X, Y, Z, U, V, W$.
These are notions from the previous works of Fu and Han \cite{FuHan2015:ISSAC}. In this subsection we recall
definitions and results about these sequences.
The most important thing to keep in mind for the sake of the proof of the Main Theorem 
is that $\bf Z$ is 
the sequence of reduced Hankel determinants modulo~2.

\medskip

The following two disjoint infinite sets of integers $J$ and $K$ 
associated with ${\bf v}$
play an important role 
in the proof of Theorem \ref{thm:main}.
\begin{defi}\label{def:JK}
Let
$P=\{1\leq i\leq d-1 \mid v_{i-1} \not= v_i\}$
and
$Q=\{1\leq i\leq d-1 \mid v_{i-1} = v_i\}$.
%
If $v_{d-1}=-1$, define
\begin{align*}
  J&=\{(dn+p)d^{2k}-1\ |\ n,k\in \NN, p\in P\}\\
   &\quad\bigcup\{(dn+q)d^{2k+1}-1\ |\ n,k\in \NN, q\in Q\},\\
  K&=\{(dn+q)d^{2k}-1\ |\ n,k\in \NN, q\in Q\}\\
   &\quad\bigcup\{(dn+p)d^{2k+1}-1\ |\ n,k\in \NN, p\in P\}.\\
\noalign{\noindent If $v_{d-1}=1$, define}
J&=\{(dn+p)d^{k}-1\ |\ n,k\in \NN, p\in P\},\\
K&=\{(dn+q)d^{k}-1\ |\ n,k\in \NN, q\in Q\}.
\end{align*}
\end{defi}
\smallskip

Let $\Sym_{m}=\Sym_{\{0,1,\ldots,m-1\}}$ be the set of all permutations on
  $\{0,1,\ldots,m-1\}$. 
For $m\ge \ell \ge 0$ let
$\mathfrak{J}_{m,\ell}$ (resp. $\mathfrak{K}_{m,\ell}$) be the set of all permutations $\sigma =\sigma_0 \sigma_1 \cdots \sigma_{m-1} \in \Sym_{m}$ such that $i+\sigma_i\in J$ (resp. $i+\sigma_i\in K$) for $i\in \{0,1,\ldots,m-1\}\setminus \{\ell\}$.
Write:
\begin{align*}
    j_{m,\ell}&:=\#\mathfrak{J}_{m,\ell},
     &k_{m,\ell}&:=\#\mathfrak{K}_{m,\ell}.\\
\end{align*}
\begin{defi}\label{def:jk}
We define the following sequences taking values in $\mathbb{Z}/2\mathbb{Z}$ :
\begin{align*}
    X_{n} & := \sum_{i=0}^{n-1}j_{n,i} \mod 2, &Y_{n}&:=j_{n,n} \mod 2,&&Z_{n}:=j_{n,n-1} \mod 2,\\
    U_{n} & := \sum_{i=0}^{n-1}k_{n,i} \mod 2, &V_{n}&:=k_{n,n} \mod 2,&&W_{n}:=k_{n,n-1} \mod 2,
\end{align*}
for $n\geq 1$, and $ X_0:=0, Y_0:=1, Z_0:=0, U_0:=0, V_0:=1, W_0:=0$.
\end{defi}
Notice that if $\ell=m$, 
then $\{0,1,\ldots, m-1\}\setminus\{\ell\}=\{0,1,\ldots, m-1\}$, so that
$j_{m,m}$ (resp. $k_{m,m}$) is the
number of permutations $\sigma\in \Sym_{m}$ such that 
$i+\sigma(i)\in J$ (resp. $\in K$) for all~$i$.
From the definitions above it is easy to see that $\NN= J\cup K$ and we have the following lemma.
\begin{lem}\label{lemma:f=set}
For each $t\ge 0$ the integer $\delta_{t}:=|(f_{t}-f_{t+1})/2|$ is equal to $1$ if and only if $t$ is in~$J$.
\end{lem}
For the proof of the the following theorem, we use the same method as in the proof of Theorem 2.2 in \cite{FuHan2015:extended}.

\begin{thm}\label{thm:J}
  Let $\bf v$ be a $\pm 1$-sequence of length~$d$ with $v_0=1$. The sequence $\bf f$ and the set $J$ associated with $\bf v$ are defined by \eqref{def:Phi} and Definition \ref{def:JK} respectively.  Then,
	the reduced Hankel determinant 
$$H_m({\bf f})/2^{m-1} \equiv Z_m \pmod 2.$$ 
\end{thm}

\begin{proof}
Let $m$ be a positive integer. By means of elementary transformations the Hankel determinant $H_{m}({\bf f})$ is equal to 
\begin{align*}
H_{m}(\bf f)
	&=
	\left|
	\begin{matrix}
	f_0 		& f_1 	& \cdots & f_{m-1}	\\
	f_1 		& f_2 	& \cdots & f_{m}		\\
	\vdots 	&\vdots & \ddots & \vdots		\\
	f_{m-1}	& f_{m}	& \cdots & f_{2m-2}
	\end{matrix}\right|\\
&=  2^{m-1} \times
 	\left|\begin{matrix}
	\frac{f_0-f_1}{2} 		& \frac{f_1-f_2}{2} 			& \cdots & \frac{f_{m-2}-f_{m-1}}{2}	& f_{m-1}		\\
	\frac{f_1-f_2}{2} 		& \frac{f_2-f_3}{2} 			& \cdots & \frac{f_{m-1}-f_{m}}{2} 		& f_{m}			\\
	\vdots 								& \vdots 									& \ddots & \vdots											& \vdots		\\
	\frac{f_{m-1}-f_m}{2}	& \frac{f_{m}-f_{m+1}}{2}	& \cdots & \frac{f_{2m-3}-f_{2m-2}}{2}& f_{2m-2}
	\end{matrix}\right|. 
\end{align*}
By Lemma \ref{lemma:f=set}, the reduced Hankel determinant is congruent modulo 2 to 
\begin{equation}\label{eqn:det}
\frac{H_{m}(\bf f)}{2^{m-1}}\equiv
\left|\begin{matrix}
\delta_0 		& \delta_1 	& \cdots & \delta_{m-2} 	& 1	  		\\
\delta_1 		& \delta_2 	& \cdots & \delta_{m-1} 	& 1				\\
\vdots 	&\vdots & \ddots & \vdots 	& \vdots	\\
\delta_{m-1}	& \delta_{m}	& \cdots & \delta_{2m-3} & 1
\end{matrix}\right|
\pmod2.
\end{equation}
By the very definition of a determinant or the Leibniz formula, the determinant
occurring on the right-hand side of the congruence \eqref{eqn:det}
is equal to
\begin{equation}\label{eqn:sumt}
	S_m:=\sum_{\sigma\in \Sym_m}(-1)^{\inv(\sigma)}\delta_{0+\sigma_0}\delta_{1+\sigma_1}\cdots \delta_{m-2+\sigma_{m-2}},
\end{equation}
where $\inv(\sigma)$ is the number of inversions of the permutation $\sigma$.
By Lemma \ref{lemma:f=set} the product $\delta_{0+\sigma_0}\delta_{1+\sigma_1}\cdots \delta_{m-2+\sigma_{m-2}}$ is equal to 1 if $i+\sigma_i\in J$ for $i=0,1,\ldots,m-2$, and to 0 otherwise. 
Hence, the summation $S_m$ 
is congruent modulo 2 to the number of permutations $\sigma\in\Sym_{m}$ such that $i+\sigma_{i}\in J$ for all $i=0,1,\ldots,m-2$, 
	which is exactly $Z_m$ by Definition \ref{def:jk}.
\end{proof}

In \cite{FuHan2015:ISSAC} an algorithm was described for {\it finding} and also {\it proving} 
a list of recurrence relations between $X_n, Y_n, Z_n, U_n, V_n, W_n$. 
The program {\tt Apwen.py} is an implementation of this algorithm in Python.
Using these relations, Fu-Han successfully proved that several sequences are Apwenian (i.e., the reduced Hankel determinants $Z_m$ are odd for all $m\in\NN$).

\medskip

\subsection{Two general Lemmas}

In this subsection we establish two general Lemmas that will be used in the proof.

For a sequence $\bf u$, we let $\sigma(\bf u)$ denote the shifted sequence 
$\sigma(\mathbf{u}):n\mapsto u(n+1)$. 
Let $S$ be a set of sequences, we let $\sigma(S)$ denote the set 
$\{\sigma(\mathbf{ u}) \mid \mathbf{u}\in S  \}$ and $\Lambda(S)$ the set
$\{\Lambda_j(\mathbf{u})\mid j=0,\ldots,d-1, \mathbf{u}\in S\}$. 
If $S$ is a set of sequences over $\mathbb{Z}/2\mathbb{Z}$, we let $\mathfrak{A}(S)$ denote the $\mathbb{Z}/2\mathbb{Z}$-algebra generated by elements of $S$.

\begin{lem}\label{lem1}
Let $S$  be a set of sequences over $\mathbb{Z}/2\mathbb{Z}$, 
if $\Lambda(S)\subseteq \mathfrak{A}(S\cup \sigma(S))$, 
then $\Lambda(\mathfrak{A}(S\cup\sigma(S)\cup\sigma^2(S)))\subseteq \mathfrak{A}
( S\cup\sigma(S)\cup\sigma^2(S))$.
\end{lem}

\begin{proof}
	Since $\Lambda(\mathfrak{A}(S\cup\sigma(S)\cup\sigma^2(S)))\subseteq \mathfrak{A}(\Lambda(S\cup\sigma(S)\cup\sigma^2(S)))$, we only need to show that $\Lambda(S\cup\sigma(S)\cup\sigma^2(S))\subseteq  \mathfrak{A}(S\cup\sigma(S)\cup\sigma^2(S))$.

	By assumption, $\Lambda(S)\subseteq \mathfrak{A}(S\cup\sigma S)\subseteq \mathfrak{A}(S\cup\sigma(S)\cup\sigma^2(S))$.

	To show that $\Lambda(\sigma(S))\subseteq \mathfrak{A}(S\cup\sigma(S)\cup\sigma^2(S))$, let $\mathbf{u}$ be an element of $S$. For $j=0,\ldots,d-2$, we have $\Lambda_j(\sigma(\mathbf{u}))=\Lambda_{j+1}(\mathbf{u})$. For $j=d-1$, we have $\Lambda_j(\sigma(\mathbf{u}))=\sigma(\Lambda_0(\mathbf{u}))$. Since $\Lambda_0(\mathbf{u})\in \mathfrak{A}(S\cup\sigma(S))$ by assumption, $\sigma(\Lambda_0(\mathbf{u}))\in \sigma(\mathfrak{A}(S\cup\sigma(S)))=\mathfrak{A}(\sigma(S)\cup\sigma^2(S))\subseteq \mathfrak{A}(S\cup\sigma(S)\cup\sigma^2(S))$.

	To show that $\Lambda(\sigma^2(S))\subseteq \mathfrak{A}(S\cup\sigma(S)\cup\sigma^2(S))$, let $\mathbf{u}$ be an element of $S$. For $j=0,\ldots,d-3$, we have $\Lambda_j(\sigma^2(\mathbf{u}))=\Lambda_{j+2}(\mathbf{u})$. Besides, $\Lambda_{d-2}(\sigma^2(\mathbf{u}))=\sigma(\Lambda_0(\mathbf{u}))$ and $\Lambda_{d-1}(\sigma^2(\mathbf{u}))=\sigma(\Lambda_1(\mathbf{u}))$. These are elements of $\mathfrak{A}(S\cup\sigma(S)\cup\sigma^2(S)$ for the same reason as before.
\end{proof}
\begin{lem}\label{lem2}
 Let $\mathbf{X}^i$, $i=1,\ldots,\ell$ be sequences taking values in a finite
 ring $R$, and $M_j$ be $\ell\times \ell$ matrices with entries in $R$ for $j=0,\ldots, d-1$.
 We denote by $\mathbf{A}$ the sequence of column vectors $A(n)=(X^1(n),\ldots,X^\ell(n))^T$. If 
 $$A(d n+j)=M_jA(n) \mbox{ for all } j=0,\ldots,d-1, \text{ and  all } n\in \mathbb{N},$$
 then $\mathbf{X}^i$ is $d$-automatic.
 \end{lem}
 \begin{proof}
  As $d$-automaticity is preserved by codings of the output alphabet, we only need to show that the sequence $\mathbf{A}$ over the finite alphabet 
  $R^\ell$ is $d$-automatic. 

By assumption, for all $j\in\{0,\ldots,d-1\}$, we have
$$M_jA=\Lambda_jA.$$
Therefore for all $i,j\in\{0,\ldots,d-1\}$, 
$$M_i M_j A= M_i \Lambda_j A=\Lambda_j M_i A=\Lambda_j\Lambda_i A.$$
By induction, for any finite word $w=w_1\cdots w_m$ over the alphabet $\{0,\ldots,d-1\}$, 
$$M_{w_1}\cdots M_{w_m}A=\Lambda_{w_m}\cdots \Lambda_{w_1}A.$$
Therefore the $d$-kernel of $\mathbf{A}$ is the set
  $$\{ M\mathbf{A} \mid M\in \{ M_0,\ldots,M_{d-1} \}^* \},$$
  which is a subset of the finite set 
  $\{M\mathbf{A} \mid M\in M_{\ell\times \ell}(R)\}$.
 \end{proof}

 \subsection{The proof}
\begin{proof}[Proof of Theorem \ref{thm:main}]
We notice that the algorithm described in \cite{FuHan2015:ISSAC} is also valid for other $\pm1$-sequences of form \eqref{def:Phi}
than Apwenian sequences,
and it will produce a list of  recurrence relations similar to those found in the two examples in
Section~\ref{sec:examples}. 
More precisely, if we denote by $S$ the set $\{{\bf X,Y,Z,U,V,W}\}$  of six sequences 
introduced in Definition \ref{def:jk}, then 
according to Theorem~4.1 of \cite{FuHan2015:ISSAC}, we have recurrence relations that
express $\Lambda_i {\bf S} $ in terms of the $\mathbb{Z}/2\mathbb{Z}$-linear combinations
of products of elements in $S\cup \sigma(S)$ for all $i=0,1,\ldots, d-1$
and all ${\bf S}\in S$. 
In other words, we have 
\begin{equation}\label{rel:S}
	\Lambda(S)\subseteq \mathfrak{A}(S\cup \sigma(S)).
\end{equation}
	According to Lemma \ref{lem1}, this implies that $\Lambda(\mathfrak{A}(S\cup\sigma(S)\cup
	\sigma^2(S)))\subseteq \mathfrak{A}(S\cup\sigma(S)\cup \sigma^2(S)))$.
Thus if we define $\bf A$ to be the vector whose components are all the
elements in
$S\cup\sigma(S)\cup\sigma^2(S)$ and all their products, then there exists
matrices
$M_0, M_1, \ldots, M_{d-1}$ that satisfy the hypothesis of Lemma \ref{lem2}.
Therefore all the components of $\bf A$, and in particular $\bf Z$, are $d$-automatic.  \end{proof}

The proofs of Lemma \ref{lem1} and Lemma \ref{lem2} being constructive, 
we may implement them to find a $d$-DFAO  that generates the sequence $\mathbf{Z}$. 
This is done by the Python program {\tt ASHankel.py} 
which takes as input the relations and initial values generated by {\tt Apwen.py}
and computes the minimal $d$-DFAO of $\mathbf{Z}$. 


\section{Two examples}\label{sec:examples}
\subsection{Example 1}\label{example1}
Take ${\bf v}=(1, 1, -1, -1, 1)$ with $d=5$ and $v_{d-1}=1$. 
Then,  the corresponding infinite $\pm1$-sequence ${\bf f}$ is equal 
to
$$\Phi(1+z-z^2-z^3+z^4)=\prod_{k\geq 0}(1+x^{5^k} - x^{2\cdot5^k}-x^{3\cdot 5^k} + x^{4\cdot5^k}).$$
The Python program {\tt Apwen.py} finds and proves the following recurrences:
\begin{align*}
X_{5n+0}&= X_{n},  &         Y_{5n+0}&= Y_n,\\
X_{5n+1}&= Z_{n+1} Y_{n},  & Y_{5n+1}&= Z_{n+1} X_n (Y_{n}+1),\\
X_{5n+2}&= 0,  &             Y_{5n+2}&= Z_{n+1} Y_{n},\\
X_{5n+3}&= 0,  &              Y_{5n+3}&= Z_{n+1} Y_{n+1},\\
	X_{5n+4}&= Z_{n+1} Y_{n+1},  & Y_{5n+4}&= Z_{n+1} X_{n+1}(Y_{n+1}+1),\\
	\noalign{\medskip}
Z_{5n+0}&= \rlap{$Z_{n}(X_n+X_nY_n+Y_n),$} &\\
Z_{5n+1}&= \rlap{$Z_{n+1}(X_n+X_nY_n+Y_n),$} &\\
Z_{5n+2}&= \rlap{$Z_{n+1} Y_n,$} &\\
Z_{5n+3}&= \rlap{$0,$} &\\
Z_{5n+4}&= \rlap{$Z_{n+1} Y_{n+1},$} &
\end{align*}
with the initial values:
$$X_0=0,\,  X_1=1,\,
	Y_0=1,\,  Y_1=0,\, 
	Z_0=0,\,  Z_1=1.
$$

	Recall that $\bf Z$ is the sequence of the reduced Hankel determinants modulo $2$.
By Theorem \ref{thm:main}, the sequence $\mathbf{Z}$ 
is $5$-automatic, and its minimal automaton is produced by the program {\tt ASHankel.py} (See Figure 1). 

\begin{figure}
\begin{tikzpicture}[shorten >=1pt,node distance=3cm,on grid,auto] 
  \node[state, initial](a) {$a$};
  \node[state] (d)[below right=of a] {$d$};
  \node[state] (b)[above right=of d] {$b$};
  \node[state] (e)[below left=of a] {$e$};
  \node[state] (c)[below right=of b] {$c$};
   \path[->]
(a) edge [loop below] node {0} ()
(a) edge [bend left=10] node {1} (b)
(a) edge [bend left=70] node {2} (c)
(a) edge [bend left] node [above]{3} (d)
(a) edge node [left] {4} (e)
(b) edge [loop below] node {0,4} ()
(b) edge node {1} (c)
(b) edge node {2} (d)
(b) edge [bend left=6] node [below] {3} (e)
(c) edge [loop right] node {0,4} ()
(c) edge node {1,2} (d)
(c) edge [bend left=50] node [above ]{3} (e)
(d) edge [loop below] node {0,1,2,3,4} ()
(e) edge [loop left] node {0,4} ()
(e) edge [bend right=80 ]  node [below] {1} (c)
(e) edge [bend right=10] node [below] {2,3} (d)
		;
\end{tikzpicture}
\caption{Minimal automaton for Example 1 where
the output function is $(a,b,c,d,e)\mapsto (0,1,1,0,0)$.}
\end{figure}
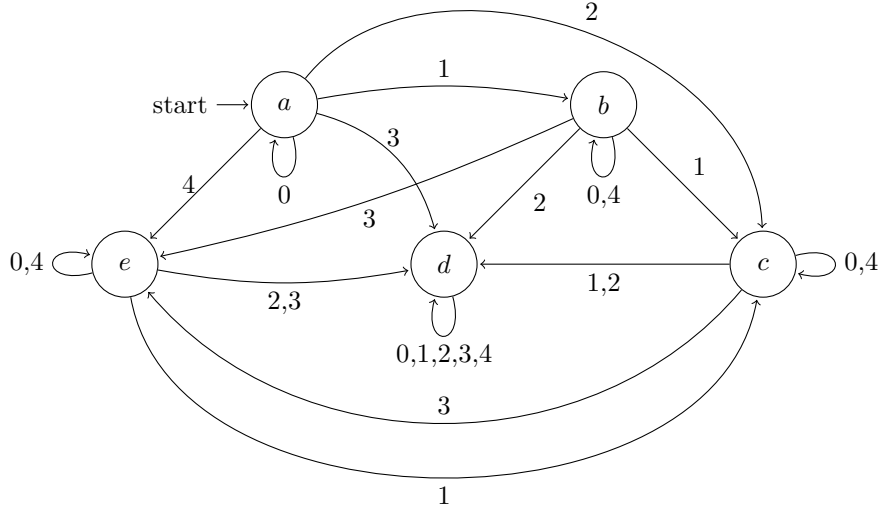

What the program {\tt ASHankel.py} does can be decomposed into the following steps: First it expands the given recurrence relations to the relations $\Lambda_i \mathbf{A}=M_i \mathbf{A}$ for $i=0,1,2,3,4$, where $M_i$ are matrices and $\mathbf{A}$ is a vector of dimension $21$ whose components are $\mathbf{X}$, $\mathbf{Y}$, $\mathbf{Z}$, $\sigma(\mathbf{X})$, $\sigma(\mathbf{Y})$, $\sigma(\mathbf{Z})$, $\sigma^2(\mathbf{X})$, $\sigma^2(\mathbf{Y})$, $\sigma^2(\mathbf{Z})$ and some of their products.  Then, by calculating all possible products of the matrices $M_i$ for $i=0,1,2,3,4$,
it finds a set $K$ of size $43$ that maps onto  the $5$-kernel of the sequence $\mathbf{A}$, as well as the transition function.
After that it reduces the set $K$ and the transition function to a smaller set $SK$ of size $9$  and the transition function $\phi:SK\times \{0,1,2,3,4\}\rightarrow SK$ by projecting each product of matrix to the row vector that corresponds to $\mathbf{Z}$. The set $SK$ maps onto the $5$-kernel of $\mathbf{Z}$ naturally. In this way we obtain a $5$-DFAO with $SK$ as the set of states, $\phi$ as the transition function, and as output function 
the map from 
$SK$ 
to the first term of the corresponding sequence in the $5$-kernel of~$\mathbf{Z}$. This $5$-DFAO generates the sequence $\mathbf{Z}$. Finally this automaton 
is converted to the minimal automaton 
of~$\mathbf{Z}$. 

In the following table, we list the first values of $H_n$ and the sequence $\bf u$ 
defined by the $5$-DFAO in Figure 1.
We verify that $u_n=H_n/2^{n-1} \mod 2=Z_n$.
\begin{table}[h]\label{table5}
	\begin{tabular}{ccccccccccccccccccc}
		\hline
		$n$ & 0 & 1 & 2 &3&4&5&6&7&8&9&10&11&12&13&14\\
		\hline
		$H_n$ & 1 & 1 & -2 &0&0&16&-32&-128&256&-1280&-6656&0&0&0&0\\
		$	u_n$ & 0 & 1 & 1 &0&0&1&1&0&0&1&1&0&0&0&0\\
		\hline
	\end{tabular}
\end{table}

\subsection{Example 2}\label{example2} 
Take ${\bf v}=(1, 1, -1, -1)$
with $d=4$ and $v_{d-1}=-1$. Then,  the corresponding infinite $\pm1$-sequence ${\bf f}$ is equal to
$$\Phi(1+z-z^2-z^3)=\prod_{k\geq 0}(1+x^{4^k} - x^{2\cdot4^k}-x^{3\cdot 4^k} ).$$
As explained in \cite{FuHan2015:extended}, 
we have six sequences $\mathbf{X}, \mathbf{Y}, \mathbf{Z}, \mathbf{U}, \mathbf{V}, \mathbf{W}$, since
$v_{d-1}=-1$.  Recall that $\mathbf{Z}$  is the sequence of the reduced Hankel determinants modulo 2.
The program {\tt Apwen.py} finds and proves the following recurrences:
\begin{align*}
	X_{4n+0}	&= 0,                       &   U_{4n+0}	&= 0							,	     \\
	X_{4n+1}	&= W_{n+1}(U_n+V_n),        &   U_{4n+1}	&= Z_{n+1} Y_n		,	      \\
	X_{4n+2}	&= 0,                       &   U_{4n+2}	&= 0							,	     \\
	X_{4n+3}	&= W_{n+1}(U_{n+1}+V_{n+1}),&   U_{4n+3}	&= Z_{n+1} Y_{n+1}	,	   \\
	\noalign{\medskip}                             \noalign{\medskip}
	Y_{4n+0}	&= U_n+V_n,                 &   V_{4n+0}	&= Y_n 			,           \\
	Y_{4n+1}	&= 0	,                      &   V_{4n+1}	&= Z_{n+1} Y_{n} 	,       \\
	Y_{4n+2}	&= W_{n+1},                 &   V_{4n+2}	&= Z_{n+1}  		,          \\
	Y_{4n+3}	&= 0	,                      &   V_{4n+3}	&= Z_{n+1} Y_{n+1} ,         \\
	\noalign{\medskip}                             \noalign{\medskip}
	Z_{4n+0}	&= W_{n} (U_n+V_n)	,         &  W_{4n+0}	&= Z_{n} Y_n	,             \\
	Z_{4n+1}	&= W_{n+1} (U_n+V_n),        &  W_{4n+1}	&= Z_{n+1} Y_n,             \\
	Z_{4n+1}	&= W_{n+1} (U_n+V_n),        &  W_{4n+2}	&= Z_{n+1} Y_n,             \\
	Z_{4n+1}	&= W_{n+1} (U_{n+1}+V_{n+1}), & W_{4n+3}	&= Z_{n+1} Y_{n+1}    , 
\end{align*}
with the initial values: 
\begin{align*}
X_0&=0,\,& Y_0&=1,\,& Z_0&=0,\,& U_0&=0,\,& V_0&=1,\,& W_0&=0,& \\
X_1&=1,\,& Y_1&=0,\,& Z_1&=1,\,& U_1&=1,\,& V_1&=1,\,& W_1&=1.& \\
\end{align*}

\begin{figure}
\begin{tikzpicture}[shorten >=1pt,node distance=3cm,on grid,auto] 
  \node[state, initial](a) {$a$};
  \node[state] (e)[below right=of a] {$e$};
  \node[state] (d)[below left=of a] {$d$};
  \node[state] (b)[above right=of e] {$b$};
  \node[state] (c)[below right=of b] {$c$};
   \path[->]
(a) edge node {0} (b)
(a) edge [bend left=70] node {1,2} (c)
(a) edge node {3} (d)
(b) edge [loop below] node {0} ()
(b) edge node [left] {1,3} (e)
(b) edge node [below] {2} (c)
(c) edge [loop right] node {0} ()
(c) edge node {1,3} (e)
(c) edge [bend left=60] node [above] {2} (d)
(d) edge node {0,2} (e)
(d) edge [bend right=90] node [below] {1} (c)
(d) edge [loop left] node {3} ()
(e) edge [loop below] node  {0,1,2,3} ()
		;
\end{tikzpicture}

\caption{Minimal automaton for Example 2
where the
output function: $(a,b,c,d,e)\mapsto (0,0,1,0,0)$.
	}
\end{figure}
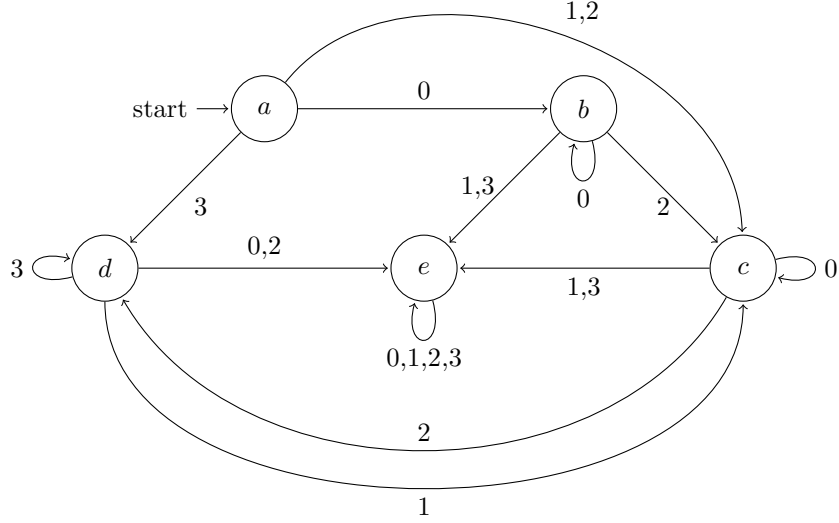
By Theorem \ref{thm:main}, the sequence ${\bf Z}$
is 4-automatic. As in the previous example, using the recurrence relations and initial values above, 
the Python program {\tt ASHankel.py} computes the minimal automaton of $\mathbf{Z}$ (See Figure 2).
In the following table, we list the first values of $H_n$ and the sequence $\bf u$ 
defined by the $4$-DFAO in Figure 2.
We verify that $u_n=H_n/2^{n-1} \mod 2=Z_n$.
\begin{table}[h]\label{table4}
	\begin{tabular}{ccccccccccccccccccccccc}
		\hline
		$n$ & 0 & 1 & 2 & 3&4&5&6&7&8&9&10&11&12&13&14&15  \\
		\hline
		$H_n$&1 & 1 & -2 & 0 &0&0&0&64&128&0&0&0&0&0&0&0 \\
		$u_n$&0 & 1 & 1 & 0 &0&0&0&1&1&0&0&0&0&0&0&0\\  
		\hline
	\end{tabular}
\end{table}


\section{Application to Irrationality Exponent}\label{section:IE}
Our result about the automaticity of the sequence $H_n/2^{n-1} \pmod 2$ enables
us to obtain upper bounds of irrationality exponents of a family of automatic
numbers.
Let $\xi$ be an irrational real number. The irrationality exponent $\mu(\xi)$ of
$\xi$ measures the approximation rate of $\xi$ by rationals. It is defined as
the supremum of the real numbers $\mu$ such that the inequality 
$$\left|\xi-\frac{p}{q}\right|<\frac{1}{q^{\mu}}$$
holds for infinitely many rational numbers $p/q$.
In this section we will make use of the following method developed 
by Bugeaud et al.
in \cite{BHWY2015}.

\begin{thm}\label{th:BHWY}
Let $d\geq 2$ be an integer, and $(c_j)_{j\geq 0}$ be an integer sequence such
that $f(z)=\sum\limits_{j=0}^{\infty} c_j z^j$ converges inside the unit disk.
Suppose that there exist integer polynomials $A(z)$, $B(z)$, $C(z)$ and $D(z)$
such that
$$f(z)=\frac{A(z)}{B(z)}+\frac{C(z)}{D(z)}f(z^d).$$
Let $b\geq 2$ be an integer such that $B(\frac{1}{b^{d^m}})C(\frac{1}{b^{d^m}})
D(\frac{1}{b^{d^m}})\neq 0$ for all integers $m>0$. If there exists an
increasing sequence of positive integers $(a_i)_{i\geq0}$ such that $H_{a_i}(f)
\neq 0$ for all integers $i\geq 0$ and $\limsup\limits_{i\rightarrow \infty}
\frac{a_{i+1}}{a_i}=\rho$, then $f(1/b)$ is transcendental, and we have
$$\mu(f(\frac{1}{b}))\leq (1+\rho)\min\{\rho^2,d\}.$$
In particular, the irrationality exponent of $f(1/b)$ is equal to $2$ if $\rho=
1$.
\end{thm}

A real number $\xi$ is said to be {\it automatic} if there exist two integers $d,b
\geq 2$ such that the $b$-ary expansion of $\xi$ is $d$-automatic.
We consider sequences of the form $\tilde{\bf{f}}=
\Phi(\tilde{\bf{v}})$ defined in \eqref{def:Phi}. This series
satisfy the functional equation
$$\tilde{\bf{f}}(x)=\tilde{\bf{v}}(x)\tilde{\bf{f}}(x^d).$$
As $\tilde{\bf{f}}$ takes only $\pm 1$ as coefficients, it
converges inside the unit disk. Moreover, the polynomial $\tilde{\bf{v}}$ does not have 
any root of the form $1/b^{d^m}$ for $b\geq 2$ and $m\geq 1$. This is because if we write
$\tilde{\bf{v}}(x)=\sum_{i=0} ^{d-1} v_i x^i$, then for any prime
factor $p$ of $b$, the $p$-adic valuation of $\tilde{\bf{v}}(1/b^{d^m})$ is
equal to the $p$-adic valuation of $(1/b^{d^m})^{d-1}$.
So we may 
apply Theorem \ref{th:BHWY} to automatic real numbers $f(\frac{1}{b}):=
\sum\limits_{j=0}^\infty \frac{f_j}{b^j}$ for $b\geq 2$.

\subsection{Two Examples}\label{ss:rho_examples}
For the series $\tilde{\bf{f}}=\Phi(1,1,-1,-1,1)$ considered in Section \ref{example1}, 
from relations between elements in the kernel of the sequence $\bf Z$, 
we deduce that the sequence $(Z_n)=(H_n/2^{n-1} \pmod 2)$ is
equal to $\tau(s^\infty(A))$ where $s$ is the $5$-substitution 
\begin{align*}
								&A\mapsto ABCDA\\
			      		 &B\mapsto BCDAB\\
			       		&C\mapsto CDDDD\\
          			&D\mapsto DDDDD,
\end{align*}
and $\tau$ is the coding $A,D\mapsto 0$, $B,C\mapsto 1$. 
Therefore,
$$\mathbf{Z}=\tau(s^n(A)s^n(B)s^n(C)s^n(D)s^n(A)s^{n+1}(B)\dots)$$
for $n\geq 1$. 
We have $s^n(C)=CD^{5^n-1}$, $s^n(D)=D^{5^n}$, and $s^n(A)$ begins by $AB$.
If we let $a_n$ denote the position of the $n$-th $1$ in the sequence $\bf Z$,
then is it easy to observe that $\limsup a_{i+1}/a_i $ is realized by the gap
between the first letter in $s^n(C)$ and the second letter in the second 
$s^n(A)$. Therefore
$$\rho=\limsup\limits_{i\rightarrow\infty} \frac{a_{i+1}}{a_i}=\lim\limits_{n 
\rightarrow \infty} 
\frac{5^n\times 4+1}{5^n\times 2}=2.$$
Theorem \ref{th:BHWY} then gives $$\mu(f(1/b))\leq 12,$$
for all integers $b\geq 2$.

\medskip

For the series $\tilde{\bf{f}}=\Phi(1,1,-1,-1)$ considered in Section
\ref{example2}, 
from relations between elements in the kernel of the sequence $\bf Z$, 
we deduce that the sequence $(Z_n=(H_n/2^{n-1} \pmod 2)$
is equal to $\tau(s^\infty(A))$ where $s$ is the $4$-substitution
\begin{align*}
	A&\mapsto ABCD\\
  B&\mapsto DDAB\\
	C&\mapsto CDDD\\
	D&\mapsto DDDD,
\end{align*}
and $\tau$ is the coding $A,D\mapsto 0$, $B,C\mapsto 1$. Therefore,
$$\mathbf{Z}=\tau(s^n(A)s^n(B)s^n(C)s^n(D)s^{n+1}(B)\dots)$$
for $n\geq 1$. We have $s^n(C)=CD^{4^n-1}$, $s^n(D)=D^{4^n}$, and 
$s^{n+1}(B)$ begins by $ D^{4^n\times 2}AB$. We let $a_n$ denote the position
of the $n$-th $1$ in the sequence $\bf Z$. We observe that $\limsup 
a_{i+1}/a_i$ is realized by the gap between the first letter in $s^n(C)$ 
and the first $B$ in $s^{n+1}(B)$. Therefore
$$\rho=\limsup\limits_{i\rightarrow \infty} \frac{a_{i+1}}{a_i}=
\lim\limits_{n\rightarrow \infty} \frac{4^{n+1}+4^n\times 2 +1}{4^n\times 2}
=3.$$
Theorem \ref{th:BHWY} then gives
$$\mu(f(1/b))\leq 16,$$
for all integers $b\geq 2$.

\subsection{The General Case and Comparison with Previous Upper Bounds}
To find an upper bound of the irrational exponent of an automatic number of 
the form $f(1/b)$ where $\tilde{\bf{f}}=\Phi(\tilde{\bf{v}})$. 
We proceed as follows. First, find a substitution $s$ and a coding $\tau$
such that $\bf Z$ is the image by $\tau$ of a fixed point of $s$;
then, calculate $\rho$ directly as in the examples;
finally, derive the irrationality exponent using Theorem \ref{th:BHWY}.
The following theorem \cite{Cobham1972} garantees that it is always possible to 
find such $s$ and $\tau$.

\begin{thm}[Cobham] 
	A sequence $\bf u$ is $d$-automatic if and only if it is the image of a fixed
	point of a $d$-substitution.
\end{thm}

The alphabet of the substitution $s$ is called the {\it internal alphabet} of 
$\bf u$.
The irrationality exponent of automatic sequences has been studied in 
\cite{AC2006}, where Adamczewski et al. established the following theorem:

\begin{thm}\label{th:AC06}
Let $d$ and $b$ be two integers at least equal to two and let $\mathbf{a}=(a_n)
_{n\geq 0}$ be an infinite sequence generated by a $d$-automaton and with values
in $\{0,1,...,b-1\}$. Let $m$ be the cardinality of the $d$-kernel of the 
sequence $\bf a$ and let $c$ be the cardinality of the internal alphabet 
associated to $\bf a$. Then, the irrationality exponent $\mu(\xi)$ of the real
number
$$\xi := \sum\limits_{n=0}^{\infty} \frac{a_n}{b_n}$$
satisfies
$$\mu(\xi)\leq cd(d^m+1).$$
\end{thm}

For certain sequences of the form $\Phi( \tilde{\bf{v}})$,
the Hankel determinant method yields a better upper bound than this 
general estimate. For example, for the sequence $\Phi(1,1,-1,-1)$ considered in 
Section \ref{ss:rho_examples}, our computation gives $\mu(f(1/b)\leq 16$ while
Theorem \ref{th:AC06} gives $\mu(f(1/b))\leq 136$; for the sequence $\Phi(1,1,-1,
-1,1)$, our computation gives $\mu(f(1/b))\leq 12$ while
Theorem \ref{th:AC06} gives $\mu(f(1/b))\leq 260$.
For the Thue-Morse sequence $\bf t$, Theorem \ref{th:AC06} gives 
$\mu(t(1/b))\leq 20$ and an argument more adapted to this specific sequence
in the same article gives $\mu(t(1/b))\leq 5$, while Hankel determinant method yields
$\mu(t(1/b))=2$,
as shown by Bugeaud in \cite{Bu2011}.
\medskip

\section{Implementation and outputs}
Our main result Theorem \ref{thm:main} says that
for each $\pm 1$-vector $\bf v$ of length
$d$, the sequence of the reduced Hankel determinants modulo 2 of the sequence
$\Phi(\tilde{\bf{v}}(x))$ is $d$-automatic.
Since the proof of the main theorem is constructive, 
we implement the method to find the minimal $d$-DFAO of the sequence 
$(H^n/2^{n-1} \pmod 2)$.
This is done by the Python programs {\tt Apwen2.py} and {\tt ASHankel.py}.
\footnote{The programs {\tt Apwen2.py} and {\tt ASHankel.py} and examples are available at \\
\texttt{irma.math.unistra.fr/\~{}guoniu/papers/w03ashankel/}}
We illustrate below the usage of these programs and their output 
for the two examples studied in Section~\ref{sec:examples} and
Section~\ref{ss:rho_examples}.

\medskip

\hrule\smallskip
\noindent
{\bf Example 1. Run the following commands} 
\smallskip
\hrule\smallskip
{
\begin{verbatim}
	> printf 'opt="StdEx"\nfrom ASHankel import *\n' > example1.py
	> python Apwen2.py 1 '1,1,-1,-1,1' >> example1.py
	> python example1.py
\end{verbatim}
}
\medskip

\goodbreak
\hrule\smallskip
\noindent
{\bf Output 1} 
\smallskip
\hrule\smallskip
\nobreak
{
\begin{verbatim}
	v= [1, 1, -1, -1, 1]
	d= 5
	Automaton:
	transition function=
	 [[0,1,2,3,4], [1,2,3,4,1], [2,3,3,4,2], [3,3,3,3,3], [4,2,3,3,4]]
	output function= [0, 1, 1, 0, 0] 

	Substitution:
	morphism= [[0,1,2,3,0], [1,2,3,0,1], [2,3,3,3,3], [3,3,3,3,3]]
	coding= [0, 1, 1, 0]
\end{verbatim}
}

\goodbreak
\medskip

\hrule\smallskip
\noindent
{\bf Example 2. Run the following commands} 
\smallskip
\hrule\smallskip
{
\begin{verbatim}
> printf 'opt="StdEx"\nfrom ASHankel import *\n' > example2.py
> python Apwen2.py 1 '1,1,-1,-1' >> example2.py
> python Apwen2.py -1 '1,1,-1,-1' >> example2.py
> python example2.py
\end{verbatim}
}
\medskip

\goodbreak
\hrule\smallskip
\noindent
{\bf Output 2} 
\smallskip
\hrule\smallskip
\nobreak
{
\begin{verbatim}
v= [1, 1, -1, -1]
d= 4
Automaton:
transition function= 
[[1,2,2,3], [1,4,2,4], [2,4,3,4], [4,2,4,3], [4,4,4,4]]
output function= [0, 0, 1, 0, 0]

Substitution:
morphism= [[0,1,2,3], [3,3,0,1], [2,3,3,3], [3,3,3,3]]
coding= [0, 1, 1, 0]
\end{verbatim}
}

\hrule\smallskip

\medskip

{\bf Acknowledgments}. The authors would like to thank 
Zhi-Ying Wen and Hao Wu
who invited them to Tsinghua University where the paper was finalized.
The authors also thank
Jean-Paul Allouche, Yann Bugeaud, 
Jacques Peyri\`ere,  Zhi-Xiong Wen, and Wen Wu for valuable discussions.


\bibliographystyle{plain}



\end{document}